\documentclass{article}
\usepackage[english]{babel}
\usepackage[utf8]{inputenc}
\usepackage[T1]{fontenc}
\usepackage{graphicx}
\usepackage{amsmath,amssymb,amsfonts,amsthm}
\usepackage{geometry}
\usepackage{mathrsfs}
\usepackage{nicefrac}
\usepackage{booktabs}
\usepackage{multirow}
\usepackage{enumitem}
\usepackage{float}
\usepackage{hyperref}
\hypersetup{colorlinks = true, linkcolor=red, citecolor=blue}

\usepackage{authblk}
\usepackage{subcaption}

\newcommand{\sumlim}[2]{\sum\limits_{#1}^{#2}}

\newcommand{\R}{\mathbb R}
\newcommand{\N}{\mathbb N}
\newcommand{\Z}{\mathbb Z}

\newcommand{\cK}{{\mathcal K}}

\newcommand{\cO}{{\mathcal O}}

\providecommand{\keywords}[1]
{
  \small	
  \textbf{Keywords.} #1
}

\providecommand{\AMS}[1]
{
  \small	
  \textbf{AMS subject classifications.} #1
}

\begin{document}

\numberwithin{equation}{section}
\newtheorem{theo}{Theorem}[section]
\newtheorem{prop}[theo]{Proposition}
\newtheorem{note}[theo]{Remark}
\newtheorem{lem}[theo]{Lemma}
\newtheorem{cor}[theo]{Corollary}
\newtheorem{definition}[theo]{Definition}
\newtheorem{assumption}{Assumption}

\title{Discretization error cancellation in the plane-wave approximation of periodic Hamiltonians with Coulomb singularities}
\author{Mi-Song Dupuy\thanks{Zentrum Mathematik, Technische Universit\"at M\"unchen, Boltzmannstra\ss e 3, 85747 Garching, Germany \\
(Email: \url{dupuy@ma.tum.de})}}


\renewcommand\Affilfont{\itshape\small}

\maketitle
\begin{abstract}
\begin{small}
In solid-state physics, energies of molecular systems are usually computed with a plane-wave discretization of Kohn-Sham equations. 
A priori estimates of plane-wave convergence for periodic Kohn-Sham calculations with pseudopotentials have been proved 
, however in most computations in practice, plane-wave cut-offs are not tight enough to target the desired accuracy. 
It is often advocated that the real quantity of interest is not the value of the energy but of energy differences for different configurations. 
The computed energy difference is believed to be much more accurate because of ``discretization error cancellation'', since the sources of numerical errors are essentially the same for different configurations. 
In the present work, we focused on periodic linear Hamiltonians with Coulomb potentials where error cancellation can be explained by the universality of the Kato cusp condition. 
Using weighted Sobolev spaces, Taylor-type expansions of the eigenfunctions are available yielding a precise characterization of this singularity.
This then gives an explicit formula of the first order term of the decay of the Fourier coefficients of the eigenfunctions.  
As a consequence, the error on the difference of discretized eigenvalues for different configurations is indeed reduced by an explicit factor. However, this error converges at the same rate as the error on the eigenvalue. Plane-wave methods for periodic Hamiltonians with Coulomb potentials are thus still inefficient to compute energy differences.
\end{small}
\end{abstract}

\keywords{Eigenvalue problems, Spectral method, Error analysis.}

\AMS{65N15, 65G99, 35P15, 65N35.}

\section*{Introduction}

In solid-state electronic structure computations, a widely used method is density functional theory, often in the form of Kohn-Sham model. In order to discretize this model, it is natural to use plane-wave expansions. 
The Kohn-Sham wave functions are known to have cusps at the positions of the nuclei. These singularities imply poor convergence rate of plane-wave methods. 
The use of pseudopotentials is designed to smooth out the cusps, hence improve plane-wave convergence. In this context, precise convergence estimates have been proved \cite{cances2012numericalPWdiscretization,chen2013numerical}. These results seem to indicate that in general, the plane-wave cut-off is not large enough to reach the desired accuracy on the computed eigenvalues.

However, when computing quantities of interest (energies, forces, response functions...) one often find errors which are much smaller than the ones predicted by the aforementioned works. 
A commonly admitted explanation for this is the fact that the computed quantities are mostly differences of energies between different (rather close) configurations. 
The sources of the numerical errors being essentially the same for different configurations, the final output is more precisely computed than the original eigenvalues \cite{pieniazek2008sources}. 
This was analyzed in a one-dimensional model in \cite{cances2017discretization}, corroborating the above argument.

The aim of the present work is to generalize the results on discretization error cancellation obtained in \cite{cances2017discretization}. The setting here is three-dimensional, with a linear Hamiltonian involving Coulomb interactions. 
We prove that numerical errors do partially cancel when computing the difference of eigenvalues of two close configurations. 
However, this difference of computed eigenvalues converges at the same rate as the computed eigenvalues. Plane-wave methods are thus not accurate enough to compute energy differences.  

The article is organized as follows. In Section~\ref{sec:setting}, we present the mathematical setting and our main result (see Theorem~\ref{thm:error_cancellation} below). In Section~\ref{sec:proofs}, we prove Theorem~\ref{thm:error_cancellation} and in Section~\ref{sec:numeric}, we present some numerical results, which are in good agreement with our theoretical results.

%

\section{Discretization error cancellation for linear Hamiltonians}
\label{sec:setting}

\subsection{Eigenvalue problem and plane-wave discretization}

Let $\Gamma = [-\tfrac{L}{2},\tfrac{L}{2}]^3$, $L>0$ be the unit cell 
repeated over a periodic lattice 
\[
\mathcal{R} = L \Z^3. 
\]
We consider the Hamiltonian $H$ acting on $L^2_\mathrm{per}(\Gamma)$ with domain $H^2_\mathrm{per}(\Gamma)$:
\begin{equation}
\label{eq:H_per}
H = -\frac{1}{2} \Delta + V_\mathrm{per} + W_\mathrm{per} .
\end{equation}
The potential $V_\mathrm{per}$ is an $\mathcal{R}$-periodic potential defined as the unique solution of 
\begin{equation}
\label{eq:Vper}
\begin{cases}
-\Delta V_\mathrm{per} = 4\pi \left( \sumlim{T \in \mathcal{R}}{} \sumlim{I=1}{N_\mathrm{at}} Z_I \left(\delta_{R_I} (\cdot + T) - \frac{1}{|\Gamma|}\right) \right) \\
  V_\mathrm{per} \text{ is } \mathcal{R}\text{-periodic}.
\end{cases}
\end{equation}
The potential $V_\mathrm{per}$ models $N_\mathrm{at} \in \N$ point charges at positions $R_I$ in the unit cell $\Gamma$ and of charge $Z_I >0$. The potential $W_\mathrm{per}$ is a smooth $\mathcal{R}$-periodic function.

The operator $H$ is self-adjoint, bounded below with compact resolvent \cite{reed1978iv}.
Thus it has a discrete spectrum of infinite eigenvalues $E_{1} \leq E_{2} \leq \dots \leq  E_{n} \to \infty$, counted with multiplicities, and the associated eigenfunctions $(\psi_{n})_{n \in \N^*}$ form an orthonormal basis of $L^2_\mathrm{per}(\Gamma)$:
\begin{equation}
\label{eq:eigenvalue_problem}
H \psi_{n} = E_{n} \psi_{n}, \quad \int_{\Gamma} \psi_n \psi_m = \delta_{nm}.
\end{equation}

The eigenvalue problem \eqref{eq:eigenvalue_problem} is solved using a plane-wave basis. Let $\mathcal{R}^*$ be the reciprocal lattice
\[
\mathcal{R}^* = \frac{2\pi}{L} \Z^3.
\]
For ${K} \in \mathcal{R}^*$, let $e_{K} = \frac{e^{i K \cdot x}}{|\Gamma|^{1/2}}$ be the plane-wave with wavevector ${K}$, where $|\Gamma|$ is the volume of the unit cell. 
The family $(e_K)_{K \in \mathcal{R}^*}$ forms an orthonormal basis of $L^2_\mathrm{per}(\Gamma)$ and for all $u \in L^2_\mathrm{per}(\Gamma)$, 
\[
u(x) = \sumlim{K \in \mathcal{R}^*}{} \widehat{u}_K e_K(x), \qquad \widehat{u}_K = \frac{1}{|\Gamma|^{1/2}} \int_{\Gamma} u(x) e^{-iK\cdot x} \, \mathrm{d}x.
\]
Since we only consider real-valued functions, the Sobolev spaces $H^s_\mathrm{per}(\Gamma)$, $s \in \R$, of real-valued $\mathcal{R}$-periodic functions are defined by
\[
H^s_\mathrm{per}(\Gamma) = \left\lbrace u(x) = \sumlim{K \in \mathcal{R}^*}{} \widehat{u}_K e_K(x) \ \Big| \ \sumlim{K \in \mathcal{R}^*}{} (1+|K|^2)^s |\widehat{u}_K|^2 < \infty, \ \widehat{u}_K^* = \widehat{u}_{-K} \right\rbrace,
\]
endowed with the inner product
\[
(u,v)_{H^s} = \sumlim{K \in \mathcal{R}^*}{} (1+|K|^2)^s \widehat{u}^*_K \widehat{v}_K.
\]
The norm associated to the inner product $(\cdot,\cdot)_{H^s}$ is denoted by $\|\cdot \|_{H^s}$.
The discretization space $V_M$, $M \in \N$, is defined by
\begin{equation}
\label{eq:discretization_space}
V_M = \left\lbrace \sumlim{|K| \leq \frac{2\pi}{L}M }{} c_K e_K \ \Big| \ c_{-K}=c_K^* \right\rbrace .
\end{equation}
The constraint $c_{-K}=c_K^*$ ensures that functions of $V_M$ are real-valued. 
\newline

The eigenvalue problem \eqref{eq:eigenvalue_problem} is solved on the Galerkin space $V_M$. The approximate eigenvalues \linebreak $E_1^M \leq E_2^M \leq \dots$ and the corresponding eigenfunctions $(\psi_i^M)$ are such that
\begin{equation}
\label{eq:discrete_eigenvalue_problem}
\forall v \in V_M, \ \int_{\Gamma} \left( -\frac{1}{2} \Delta + V \right) \psi_i^M  v = E_i^M \int_\Gamma \psi_i^M v, \ \int_\Gamma \psi_i^M \psi_j^M = \delta_{ij}.
\end{equation}

The bilinear form associated to the operator $H$ is $H^1$-bounded and coercive. Hence if $M$ is sufficiently large, the eigenpairs $(E_i^M,\psi_i^M)$ of the variational approximation \eqref{eq:discrete_eigenvalue_problem} satisfy \cite{babuska1989finite}
\begin{equation}
0 \leq E_i^M - E_i \leq C \sup\limits_{\substack{\psi \in M(E_i) \\ \|\psi\|_{H^1}=1}} \inf_{v \in V_M} \|\psi - v \|_{H^1}^2,
\end{equation}
and there exists $\psi_i \in M(E_i)$ ($M(E_i)$ is the vector space of eigenfunctions of \eqref{eq:eigenvalue_problem} associated to the eigenvalue $E_i$) such that
\begin{equation}
\label{eq:exact_eigenfunction_correspondance}
\| \psi_i^M - \psi_i \|_{H^1} \leq C\sup\limits_{\substack{\psi \in M(E_i) \\ \|\psi\|_{H^1}=1}} \inf_{v \in V_M} \|\psi - v \|_{H^1}.
\end{equation}

By Sobolev embedding theorem, $H_\mathrm{per}^{3/2+\varepsilon}(\Gamma) \hookrightarrow C^0_\mathrm{per}(\Gamma) $ for all $\varepsilon >0$, hence $\sumlim{{T} \in \mathcal{R}}{} \delta_{{R}_I}(\cdot + {T}) \in H^{-3/2-\varepsilon}_\mathrm{per}(\Gamma)$. Thus $V_\mathrm{per}$ defined by \eqref{eq:Vper} belongs to $H^{1/2-\varepsilon}_\mathrm{per}(\Gamma)$ for all $\varepsilon>0$.
Using elliptic regularity, eigenfunctions $\psi_n$ of the eigenvalue problem \eqref{eq:eigenvalue_problem} belongs to $H_\mathrm{per}^{5/2-\varepsilon}(\Gamma)$ for all $\varepsilon >0$.

Let $\Pi_M$ be the $L^2$-orthogonal projector onto $V_M$. Since $(e_K)_{K \in \mathcal{R}^*}$ is an orthogonal basis of $H^s_\mathrm{per}(\Gamma)$ for all $s \in \R$ (if $s=0$, $H^0_\mathrm{per}(\Gamma)=L^2_\mathrm{per}(\Gamma)$), the best approximation of $\psi \in H^s_\mathrm{per}(\Gamma)$ in $V_M$ is simply $\Pi_M \psi$.
Hence using that for all $r,s \in \R$ with $r \leq s$, we have for $f \in H^s_\mathrm{per}(\Gamma)$
\begin{equation}
\label{eq:best_approximation_Hs}
\|f - \Pi_M f\|_{H^r} \leq \left( \frac{L}{2\pi M} \right)^{s-r} \|f - \Pi_M f\|_{H^s},
\end{equation}
we deduce that for all $\varepsilon >0$,
\begin{equation}
\label{eq:eigenvalue_convergence}
0 < E_i^M - E_i \leq C \| \psi_i - \Pi_M \psi_i \|_{H^1}^2 \leq \frac{C}{M^{3-\varepsilon}} \|\psi\|_{H^{(5-\varepsilon)/2}}^2.
\end{equation}

The goal of this paper is to give an explicit expression of the first order term in \eqref{eq:eigenvalue_convergence} using the particular nature of the singularity of the eigenfunctions $\psi$ in \eqref{eq:eigenvalue_problem}.
Weighted Sobolev space and singular expansion of these eigenfunctions constitute the appropriate way to characterize precisely such singularities.

\subsection{Singular expansion}

The theory of weighted Sobolev spaces has been introduced to study singularities of boundary value problems in conical domains with corners and edges \cite{babuska1972finite,grisvard1992singularities}.  
It is also closely linked to the $b$-calculus of pseudodifferential operators developed by Melrose \cite{melrose93}. 
Although the geometry here is simple, Coulomb singularities generated by the nuclei fit perfectly in this treatment.
The behavior of the electronic wave function close the nucleus has been precisely characterized using this theory \cite{flad2008asymptotic,hunsicker2008analysis}. 
Those results paved the way to the analysis of the muffin-tin and LAPW methods \cite{chen2015numerical} and the VPAW method \cite{dupuy2018vpaw3d}. 
The interested reader may refer to \cite{kozlov1997elliptic, egorov2012pseudo} for a detailed exposition of this theory. 
We briefly expose the definition of the weighted Sobolev space in our setting and some important results used to prove Theorem~\ref{thm:error_cancellation}.\\ 

Let $\mathcal{S}$ be the set of the positions of the nuclei:
\[
\mathcal{S} = \lbrace R_I + T, \ I = 1,\dots,N_\mathrm{at}, \ T \in \mathcal{R} \rbrace.
\]
Let $\varrho$ be a $\mathcal{R}$-periodic continuous function such that $\varrho(R_I + x) = |x|$ for small $x$, $\varrho \in C^\infty_\mathrm{loc}(\mathbb{R}^3 \setminus \mathcal{S})$.

\begin{definition}
Let $k \in \N$ and $\gamma \in \R$.
We define the $k$-th weighted Sobolev space with index $\gamma$ by
\begin{equation}
\mathcal{K}^{k,\gamma}(\Gamma) = \left\lbrace u \in L^2_\mathrm{per}(\Gamma) : \varrho^{|\alpha|-\gamma} \partial^\alpha u \in L^2_\mathrm{per}(\Gamma) \ \forall \ |\alpha | \leq k \right\rbrace.
\end{equation}
\end{definition}

Consider a subspace of functions with the asymptotic expansions
\begin{equation}
\label{eq:weighted_sobolev_asymptotic}
\forall \, I =1,\dots,N_\mathrm{at}, \ u(R_I+ x) \sim \sumlim{j \in \N}{} c_j(\widehat{x}) |x|^j \ \text{as } x \to 0,
\end{equation}
where $c_j$ belongs to the finite dimensional subspace $M_j = \mathrm{span} \lbrace Y_{\ell m}, 0 \leq \ell \leq j, |m| \leq \ell \rbrace$ and for $x\in \R^3$, $\widehat{x} = \tfrac{x}{|x|}$. 

We define the weighted Sobolev spaces with asymptotic type \eqref{eq:weighted_sobolev_asymptotic}:
\begin{equation}
\label{eq:weighted_sobolev_inVPAW3D}
\begin{split}
\mathscr{K}^{k,\gamma}(\Gamma) = \Bigg\lbrace u \in \mathcal{K}^{k,\gamma}(\Gamma) \, \bigg| \, \eta_N \in \mathcal{K}^{k,\gamma+N+1}(\Gamma) \text{ where } \eta_N \text{ is the } \Gamma\text{-periodic function defined in } \Gamma \text{ by} \\
\left. \forall N \in \N, \ \forall x \in \Gamma, \ \eta_N(x) = u(x) - \sumlim{I=1}{N_\mathrm{at}} \omega(|x-R_I|) \sumlim{j=0}{N} c_j^I(\widehat{x-R_I}) |x-R_I|^j \right\rbrace,
\end{split}
\end{equation}
where $\omega$ is a smooth positive cutoff function, \emph{i.e.} $\omega = 1$ near 0 and $\omega = 0$ outside some neighbourhood of $0$. 

The definition \eqref{eq:weighted_sobolev_inVPAW3D} slightly differs from the definition of the weighted Sobolev space given in \cite{chen2015numerical} (Equation (2.6)). However, our definition is consistent with the results that can be found in \cite{hunsicker2008analysis} (see Theorem I.1) and the original paper \cite{flad2008asymptotic} (see Proposition 1) from which the definition appearing in \cite{chen2015numerical} is taken. 

The expansion \eqref{eq:weighted_sobolev_asymptotic} can be viewed as a ``regularity expansion''. Let us suppose that the functions $c_j$ in the singular expansion are constant. Then all the even terms appearing in \eqref{eq:weighted_sobolev_inVPAW3D} are smooth since for any $k \in \N$, $x \mapsto |x|^{2k}$ is smooth. 
For the odd terms in the expansion, the function $x \mapsto |x|$ is continuous but not differentiable at the origin, the function $x \mapsto |x|^3$ is $C^2$ but not $C^3$ and so on. 
Since the decay of the Fourier coefficients depends on the regularity of the function, this expansion enables one to characterize precisely this decay.

The following result, stated in \cite{hunsicker2008analysis,chen2015numerical} (see also \cite{flad2008asymptotic} for similar results in the Hartree-Fock model), gives the regularity of the eigenfunction of \eqref{eq:eigenvalue_problem} in terms of the previously defined weighted Sobolev space.

\begin{theo}
\label{thm:psi_well_behaved_periodic}
Let $\psi$ be an eigenfunction of $H \psi = E\psi$ where $H$ is defined in \eqref{eq:H_per}. Then $\psi$ belongs to $\mathscr{K}^{\infty,\gamma}(\Gamma)$ for all $\gamma < \frac{3}{2}$. The first two terms of the asymptotic expansion \eqref{eq:weighted_sobolev_asymptotic} are explicit and given by
\begin{equation}
c_0^I = \psi(R_I), \quad \quad c_1^I= -Z_I \psi(R_I) + \sumlim{m=-1}{1} \alpha_m Y_{1m}, \ \alpha_m \in \R.
\end{equation}
\end{theo}

In \cite{flad2008asymptotic,chen2015numerical}, functions belonging to $\mathscr{K}^{\infty,\gamma}(\Gamma)$ are called ``well-behaved''.
It is easy to see that if $u$ is asymptotically well-behaved then by the definition of the weighted Sobolev space with asymptotic type \eqref{eq:weighted_sobolev_asymptotic}, the remainder
$\eta_N$ is in the classical Sobolev space $H^{5/2 + N -\varepsilon}_\mathrm{per}(\Gamma)$. 

The last assertion is the well-known Kato cusp condition \cite{kato1957eigenfunctions}.

\subsection{Main result}

Using the previous characterization of the eigenfunctions of the periodic Hamiltonian \eqref{eq:H_per}, an explicit expression of the error cancellation factor can be obtained. The proof of the next theorem can be found in Section~\ref{sec:proofs}. 

\begin{theo}
\label{thm:error_cancellation}
Let $(\psi_i,E_i)$ be an eigenpair of the operator $H$ defined in \eqref{eq:H_per}. Let $E_i^M$ be the corresponding eigenvalue obtained by the plane-wave variational approximation on $V_M$ given by \eqref{eq:discretization_space}.
Let $M_0 >0$ sufficiently large such that \eqref{eq:exact_eigenfunction_correspondance} holds for all $M \geq M_0$. Then for all $\varepsilon>0$, there exists a positive constant $C_{\varepsilon,M_0}$ such that for all $M \geq M_0$, we have
\begin{equation}
\label{eq:final_error_estimate}
\left| E_i^M - E_i - \frac{2L^{3}}{3\pi^{3}M^3} \sumlim{I=1}{N_\mathrm{at}} Z_I^2 \psi_i(R_I)^2 \right| \leq  \frac{C_{\varepsilon,M_0}}{M^{4-\varepsilon}}.
\end{equation}
\end{theo}

For two different atomic configurations, $\mathbf{R}^{(1)} = (R_1^{(1)},\dots,R_{N_\mathrm{at}}^{(1)})$ and $\mathbf{R}^{(2)} = (R_1^{(2)},\dots,R_{N_\mathrm{at}}^{(2)})$, the error on the discretized eigenvalue difference is
\[
E_i^{\mathbf{R}_1,M} - E_i^{\mathbf{R}_2,M} = E_i^{\mathbf{R}_1} - E_i^{\mathbf{R}_2} + \frac{2L^{3}}{3\pi^{3}M^3} \sumlim{I=1}{N_\mathrm{at}} Z_I^2 \left(\psi^{(1)}\left(R_I^{(1)} \right)^2- \psi^{(2)}\left(R_I^{(2)}\right)^2\right)+ \cO \left( \frac{1}{M^{4-\varepsilon}} \right),
\]
where $\psi^{(1)}$ and $\psi^{(2)}$ are $L^2$-normalized eigenfunctions associated respectively to $E_i^{\mathbf{R}_1}$ and $E_i^{\mathbf{R}_2}$.
Since there is a differentiable dependence on the atomic positions, the error cancellation is of order  $\frac{|\mathbf{R}^{(1)}-\mathbf{R}^{(2)}|}{M^3}$. The convergence rate of the eigenvalue difference is the same as the eigenvalue error, however the prefactor is reduced (see Figure~\ref{fig:error_cancellation} for an example on a simple model). 
\newline

In \cite{cances2017discretization}, the authors analyzed the phenomenon of error cancellation in the case of the lowest eigenvalue of the periodic one-dimensional Schr\"odinger operator
\[
H = -\frac{\mathrm{d}^2}{\mathrm{d}x^2} - Z_0 \sumlim{k \in \Z}{} \delta_k - Z_R \sumlim{k \in \Z}{} \delta_{k+R}, \quad 0 <R<1, \quad Z_0,Z_R >0.
\]
The authors showed that the convergence of the lowest eigenvalue $E_M$ computed with plane-waves with wavenumber $|k|\leq M$ is given by
\[
E_M = E + \frac{Z_0^2 \psi(0)^2+Z_R^2 \psi(R)^2}{2\pi^2 M} + \cO \left( \frac{1}{M^{2-\varepsilon}} \right) ,
\]
for all $\varepsilon >0$ and $M$ sufficiently large. The function $\psi$ is an $L^2$-normalized eigenfunction associated to the eigenvalue $E$. 
It is interesting to notice the similarity with the expression \eqref{eq:final_error_estimate} obtained here. 
This stems from the fact that in both models, the singularities of the eigenfunctions are comparable: eigenfunctions are Lipschitz at the positions of the nuclei but generally not differentiable.
\newline

Equation~\eqref{eq:final_error_estimate} gives a first order correction formula to the computed eigenvalue $E_i^M$. 
Since $\psi_i \in H^{5/2-\varepsilon}_\mathrm{per}(\Gamma)$ and $H^{3/2+\varepsilon}_\mathrm{per}(\Gamma) \hookrightarrow C^0_\mathrm{per}(\Gamma)$ for all $\varepsilon >0$, by Equations \eqref{eq:exact_eigenfunction_correspondance} and \eqref{eq:best_approximation_Hs}, for all $\varepsilon>0$, there is a constant $C_\varepsilon$ such that $\| \psi_i^M - \psi_i \|_{L^\infty} \leq C \| \psi_i^M - \psi_i \|_{H^{3/2+\varepsilon/2}} \leq  \frac{C_\varepsilon}{M^{1-\varepsilon}}$. 
Thus the error estimate~\eqref{eq:final_error_estimate} can be written,
\begin{equation}
\label{eq:first_order_correction}
E_i^M - E_i = \frac{2L^{3}}{3\pi^{3}M^3} \sumlim{I=1}{N_\mathrm{at}} Z_I^2 \psi_i^M(R_I)^2 + \cO \left( \frac{1}{M^{4-\varepsilon}} \right).
\end{equation}
Computing the first order correction only requires an inverse FFT to get the value of the wave function at the positions of the nuclei. 
This correction improves the convergence rate by a factor $\frac{1}{M}$. Other approaches dealing directly (see \cite{chen2015numerical,dupuy2018vpaw3d}) or indirectly (\emph{e.g.} using pseudopotentials \cite{cances2012numericalPWdiscretization,cances2016perturbation}) with the Coulomb singularities exist and are more efficient. 
\newline

The numerical study in \cite{cances2017discretization} suggests that error cancellation also happens in periodic Kohn-Sham computations with pseudopotentials. 
In these models, under some assumptions on the regularity of the exchange-correlation potential and the positivity of the whole electronic density (frozen-core and valence electron densities), the eigenfunctions are smooth in the whole domain except at some spheres centered at the positions of the nuclei. 
These spheres correspond to the regions where the pseudopotentials are not smooth, because of a mismatch of higher derivatives with the true electronic potential. With a precise characterization of the singularity induced by the pseudopotentials, a similar analysis of the discretization error cancellation in plane-wave calculations should be possible.

\section{Proof of Theorem~\ref{thm:error_cancellation}}
\label{sec:proofs}

In this section, $M$ is an integer large enough so that \eqref{eq:exact_eigenfunction_correspondance} holds. 
$\Pi_M$ denotes the $L^2$-orthogonal projector onto $V_M$ defined in \eqref{eq:discretization_space}. 
We denote by $\Pi_M^\perp = \mathrm{Id} - \Pi_M$. 

\subsection{Estimates on the Fourier coefficients}

\begin{lem}
\label{lem:basic_estimate_psi_M}
Let $\psi$ be an $L^2$-normalized eigenfunction of \eqref{eq:eigenvalue_problem}. 
\begin{enumerate}
\item Then for $s<\tfrac{5}{2}$ and all $\varepsilon >0$, we have a positive constant $C_{s,\varepsilon}$ independent of $M$ such that 
\begin{equation}
\| \Pi_M^\perp \psi \|_{H^s} \leq \frac{C_{s,\varepsilon}}{M^{5/2-s-\varepsilon}}.
\end{equation}
\item Let $\psi_i^M$ be an eigenfunction of the variational approximation \eqref{eq:discrete_eigenvalue_problem}. Let $\psi_i$ be the corresponding eigenfunction such that \eqref{eq:exact_eigenfunction_correspondance} holds. Then for $s<\tfrac{5}{2}$ and all $\varepsilon >0$, we have a positive constant $C_{s,\varepsilon}$ independent of $M$ such that 
\begin{equation}
\| \psi_i^M - \psi_i \|_{H^s} \leq \frac{C_{s,\varepsilon}}{M^{5/2-s-\varepsilon}}.
\end{equation}
\end{enumerate}
\end{lem}

\begin{proof}
This lemma is proved by noticing that $\psi \in H^{5/2-\varepsilon}_\mathrm{per}(\Gamma)$ and using \eqref{eq:best_approximation_Hs}.
\end{proof}


\begin{lem}
\label{lem:psi_expansion_in_M}
Let $K \in \mathcal{R}^*$ and let $\widehat{\eta}_K$ be the Fourier coefficient for the wavenumber $K$ of the remainder of the singular expansion \eqref{eq:weighted_sobolev_inVPAW3D} for $N=1$. Then, for any $n \in \N$, we have
\begin{equation}
\label{eq:fourier_coefficient_asymptotic}
\widehat{\psi}_K = \frac{1}{|\Gamma|^{1/2}} \int_\Gamma \psi(x) e^{-iK\cdot x}\, \mathrm{d}x = \frac{8{\pi}}{|\Gamma|^{1/2} |K|^4} \sumlim{I=1}{N_\mathrm{at}} Z_I \psi(R_I) e^{-i K\cdot R_I} + \widehat{\eta}_K + o \left( \frac{1}{K^n} \right).
\end{equation}
\end{lem}
\begin{proof}
Let $\eta$ be the remainder of the singular expansion of $\psi$ for $N=1$. Hence, $\eta \in \cK^{\infty, \frac{7}{2}-\varepsilon}(\Gamma)$ for all $\varepsilon >0$. Using Theorem \ref{thm:psi_well_behaved_periodic} and noticing that $|x-R_I| Y_{1m}(\widehat{x-R_I})$ is a polynomial hence a smooth function,  we have
\begin{align*}
\int_\Gamma \psi(x) e^{-iK\cdot x}\, \mathrm{d}x & = \int_{\Gamma} \sumlim{I=1}{N_\mathrm{at}} \omega(|x-R_I|) \sumlim{j=0}{1} c_j^I(\widehat{x-R_I}) |x-R_I|^j e^{-i K \cdot x} \, \mathrm{d} x + \int_{\Gamma} \eta(x) e^{-iK\cdot x} \, \mathrm{d}x \\
& = -\int_{\Gamma} \sumlim{I=1}{N_\mathrm{at}} \omega(|x-R_I|) Z_I \psi(R_I) |x-R_I|  e^{-i K \cdot x} \, \mathrm{d} x + \int_{\Gamma} \eta(x) e^{-iK\cdot x} \, \mathrm{d}x + o\left( \frac{1}{K^n} \right) .
\end{align*}
$\omega$ is a smooth function such that $\omega = 1$ near 0 and $0$ outside a neighbourhood of 0, hence for each \linebreak $I=1,\dots,N_\mathrm{at}$, we have
\[
\int_\Gamma \omega(|x-R_I|) |x-R_I|  e^{-i K \cdot x} \, \mathrm{d} x = e^{-i K \cdot R_I} \int_{B(0,R)} \omega(|x|) |x| e^{-i K \cdot x} \, \mathrm{d}x.
\]
Written in spherical coordinates and using the radial symmetry to replace $K \cdot x$ by $Kr \cos(\theta)$, this yields
\[
\int_\Gamma \omega(|x-R_I|) |x-R_I|  e^{-i K \cdot x} \, \mathrm{d} x = e^{-i K \cdot R_I} \int_{B(0,R)} \omega(r) e^{-i|K|r\cos(\theta)} r^3 \sin(\theta) \, \mathrm{d}\phi \, \mathrm{d}\theta \, \mathrm{d}r.
\]
Thus we have
\begin{align*}
\int_{B(0,R)} \omega(r) e^{-i|K|r\cos(\theta)} r^3 \sin(\theta) \, \mathrm{d}\phi \, \mathrm{d}\theta \, \mathrm{d}r & = 2 \pi \int_{0}^R \omega(r) r^3 \int_0^\pi e^{-i|K|r\cos(\theta)} \, \mathrm{d} \theta \, \mathrm{d}r \\
& = \frac{2\pi}{i|K|} \int_0^R \omega(r) \left[ e^{-i|K|r\cos(\theta)} \right]_0^\pi r^2 \, \mathrm{d} r \\
& = \frac{4\pi}{|K|} \int_0^R \omega(r) \sin(|K|r) r^2 \, \mathrm{d} r.
\end{align*}
Using successive integration by parts and noticing that $\omega' \in C^\infty_c(0,R)$, we have for any $n \in \N$
\[
\int_{B(0,R)} \omega(r) e^{-i|K|r\cos(\theta)} r^3 \sin(\theta) \, \mathrm{d}\phi \, \mathrm{d}\theta \, \mathrm{d}r = -\frac{8\pi}{|K|^4} + o\left( \frac{1}{|K|^n} \right).
\]
\end{proof}

\subsection{Error estimates}

\begin{lem}
\label{lem:E_M-E}
Let $\psi_i^M$ be the $i$-th eigenvalue $E_i^M$ of the variational approximation \eqref{eq:discrete_eigenvalue_problem}. Let $\psi_i$ be the corresponding eigenfunction associated to the eigenvalue $E_i$ such that \eqref{eq:exact_eigenfunction_correspondance} holds.

The error on the discretized eigenvalue is given by
\begin{equation}
\label{eq:error_formula}
E_i^M - E_i = -\left(\int_{\Gamma} V \Pi_M^\perp \psi_i \psi_i^M \right) \left(1 + \cO \left( \frac{1}{M^{5/2-\varepsilon}} \right)\right).
\end{equation}
\end{lem}

\begin{proof}
Let $\psi_i^M$ be an $L^2$-normalized eigenfunction of the variational approximation \eqref{eq:discrete_eigenvalue_problem} satisfying
\begin{equation}
\label{eq:psi_iM}
\forall f \in V_M, \ \int_{\Gamma} (-\frac{1}{2} \Delta \psi_j^M + V \psi_j^M)f  = E_j^M \int_{\Gamma} \psi_j^M f, \quad \int_{\Gamma} \psi_j^M \psi_k^M = \delta_{jk}. 
\end{equation}
Let $\psi_i$ be the corresponding eigenfunction such that  \eqref{eq:exact_eigenfunction_correspondance} is satisfied. For this eigenfunction, we have
\begin{equation}
\label{eq:psi_i}
\forall f \in L^2(\Gamma), \ \int_{\Gamma} (-\frac{1}{2} \Delta \psi_j + V \psi_j)f  = E_j \int_{\Gamma} \psi_j f. 
\end{equation}
Taking $f = \Pi_M  \psi_i$ in \eqref{eq:psi_iM} and $f = \psi_i^M$ in \eqref{eq:psi_i} and subtracting both equations, we obtain 
\begin{equation*}
E_i^M \int_{\Gamma} \psi_i^M \Pi_M \psi_i - E_i \int_{\Gamma} \psi_i \psi_i^M = \int_{\Gamma}(-\frac{1}{2} \Delta \psi_i^M + V \psi_i^M) \Pi_M \psi_i - \int_\Gamma   (-\frac{1}{2} \Delta \psi_i + V \psi_i) \psi_i^M.
\end{equation*}
Since the plane-waves $(e_K)_{K \in \mathcal{R}^*}$ are orthogonal in $L^2_\mathrm{per}(\Gamma)$ and in $H^{1}_\mathrm{per(\Gamma)}$, we have
\begin{equation*}
\int_{\Gamma} \psi_i^M \Pi_M \psi_i = \int_{\Gamma} \psi_i^M  \psi_i, \qquad \text{and} \qquad \int_{\Gamma}-\frac{1}{2} \Delta \psi_i^M \Pi_M \psi_i = \int_{\Gamma}-\frac{1}{2} \Delta \psi_i^M \psi_i.
\end{equation*}
Thus, we have
\begin{align*}
(E_i^M - E_i)\int_{\Gamma} \psi_i^M  \psi_i  = \int_\Gamma V \Pi_M \psi_i \psi_i^M - \int_\Gamma V \psi_i \psi_i^M = -\int_\Gamma V \Pi_M^\perp \psi_i \psi_i^M .
\end{align*}
Since for all $\varepsilon >0$, $\psi_i \in H^{5/2-\varepsilon}_\mathrm{per}(\Gamma)$ and is such that \eqref{eq:exact_eigenfunction_correspondance} holds, we have
\begin{equation*}
\int_{\Gamma} \psi_i^M \psi_i = \int_{\Gamma} |\psi_i^M|^2 - \int_\Gamma \psi_i^M (\psi_i - \psi_i^M) = 1 + \cO\left(\frac{1}{M^{5/2-\varepsilon}} \right).
\end{equation*}
\end{proof}

Since Equations \eqref{eq:fourier_coefficient_asymptotic} and \eqref{eq:error_formula} do not involve other energy levels, we drop the index $i$ in the remainder of the proof. 

\begin{lem}
\label{lem:remainder_explicited}
Let $M \in \N$ and $V$ the potential of the periodic Hamiltonian in \eqref{eq:H_per}. Then
\begin{equation}
\int_{\Gamma} V \Pi_M^\perp \psi \psi^M = -\frac{4 \pi}{|\Gamma|} \sumlim{I=1}{N_\mathrm{at}} \sumlim{|K'|> \frac{2 \pi}{L} M}{} \sumlim{|K| \leq \frac{2 \pi}{L} M}{} Z_I \frac{e^{-i K \cdot R_I} e^{iK' \cdot R_I}}{|K'-K|^2} \widehat{\psi}_K^*\widehat{\psi}_{K'} + \cO \left( \frac{1}{M^{9/2-\varepsilon}} \right).
\end{equation}
\end{lem}

\begin{proof}
We have
\begin{align*}
\int_\Gamma V \Pi_M^\perp \psi \psi_M &= \int_\Gamma (V_\mathrm{per} + W_\mathrm{per}) \Pi_M^\perp \psi \psi_M \\
&= \int_{\Gamma} V_\mathrm{per} \Pi_M^\perp \psi \Pi_M \psi + \int_{\Gamma} W_\mathrm{per} \Pi_M^\perp \psi \Pi_M\psi + \int_{\Gamma} V \Pi_M^\perp \psi (\psi_M - \Pi_M \psi ).
\end{align*}
We first bound the last two terms. Using Lemma \ref{lem:basic_estimate_psi_M}, we have
\begin{align*}
\left| \int_{\Gamma} V \Pi_M^\perp \psi (\psi_M - \Pi_M \psi) \right| & \leq \left\| V \Pi_M^\perp \psi \right\|_{L^2} \left\| \psi_M - \Pi_M \psi \right\|_{L^2} \\
& \leq \frac{C_\varepsilon}{M^{5/2-\varepsilon/2}} \|V\|_{H^{-1/2-\varepsilon/2}} \|\Pi_M^\perp \psi\|_{H^{1/2+\varepsilon/2}} \\
& \leq \frac{C_\varepsilon}{M^{9/2-\varepsilon}},
\end{align*}
for some positive constant $C_\varepsilon$ independent of $M$. 
Since $W_\mathrm{per}$ is smooth, we have for any $n \in \N^*$, a positive constant $C_n$ independent of $M$ such that
\begin{align*}
\left| \int_{\Gamma} W_\mathrm{per} \Pi_M^\perp \psi \Pi_M \psi \right| & \leq \|\Pi_M W_\mathrm{per} \Pi_M^\perp\|_{L^2} \|\Pi_M^\perp \psi\|_{L^2} \|\Pi_M\psi \|_{L^2} \\
& \leq \frac{C_n}{M^n}.
\end{align*}
Finally since $V_\mathrm{per}$ is defined by \eqref{eq:Vper}, we have
\begin{align*}
\int_{\Gamma} V_\mathrm{per} \Pi_M^\perp \psi \Pi_M \psi & =- \frac{4\pi}{|\Gamma|} \sumlim{I=1}{N_\mathrm{at}} 
Z_I \int_{\Gamma} \sumlim{K \not= 0}{}  \frac{e^{-iK \cdot R_I}}{|K|^2} e^{iK\cdot x} \frac{1}{|\Gamma|^{1/2}}
\sumlim{|K'|> \frac{2\pi}{L} M}{} \widehat{\psi}_{K'} e^{i K' \cdot x} \frac{1}{|\Gamma|^{1/2}} \sumlim{|K''| \leq \frac{2\pi}{L} M}{} \widehat{\psi}_{K''} e^{iK'' \cdot x} \, \mathrm{d}x \\
& =- \frac{4\pi}{|\Gamma|^2} \sumlim{I=1}{N_\mathrm{at}} 
Z_I \int_{\Gamma} \sumlim{|K'| > \frac{2\pi}{L}M}{} \sumlim{K \not=K'}{} \frac{e^{-iK\cdot R_I} e^{iK'\cdot R_I}}{|K-K'|^2} \widehat{\psi}_{K'} e^{i K\cdot x} \sumlim{|K''| \leq \frac{2\pi}{L} M}{} \widehat{\psi}_{K''} e^{iK'' \cdot x} \, \mathrm{d}x \\
& =- \frac{4\pi}{|\Gamma|}  \sumlim{I=1}{N_\mathrm{at}} \sumlim{|K'|> \frac{2 \pi}{L} M}{} \sumlim{|K| \leq \frac{2 \pi}{L} M}{} Z_I \frac{e^{-i K \cdot R_I} e^{iK' \cdot R_I}}{|K'-K|^2} \widehat{\psi}_K^* \widehat{\psi}_{K'} ,
\end{align*}
where we used that $\int_{\Gamma} e^{iK\cdot x} e^{iK''\cdot x} \, \mathrm{d}x = |\Gamma| \delta_{K+K''}$ and $\widehat{\psi}_{-K''} = \widehat{\psi}_{K''}^*$ by definition of the variational space $V_M$.
\end{proof}

\begin{lem}
\label{lem:double_sum_split}
Let $M_0$ be a positive constant. Then for all $\varepsilon>0$ and $1 \leq I \leq N_\mathrm{at}$, there exists a constant $C_{\varepsilon, M_0}$ such that for all $M \geq M_0$ we have 
\begin{multline} 
\left| \sumlim{|K'|>\frac{2\pi}{L}M}{} \sumlim{|K|\leq \frac{2\pi}{L}M}{} \frac{e^{-iK \cdot R_I}\widehat{\psi}_K^* e^{i K' \cdot R_I }\widehat{\psi}_{K'} }{ |K-K'|^2} - \frac{8 {\pi}}{|\Gamma|^{1/2}} \sumlim{J=1}{N_\mathrm{at}} Z_J  \sumlim{|K|\leq \frac{2\pi}{L}M}{} \widehat{\psi}_K^* e^{-iK \cdot R_I}  \sumlim{|K'| > \frac{2\pi}{L}M}{} \frac{e^{i K' \cdot(R_I-R_J) } }{|K'|^6 } \right| 
\leq  \frac{C_{\varepsilon,M_0}}{M^{4-\varepsilon}} .
\end{multline}
\end{lem}

\begin{proof}
By Lemma \ref{lem:psi_expansion_in_M}, we have
\begin{multline*}
\sumlim{|K'|>\frac{2\pi}{L}M}{} \sumlim{|K|\leq \frac{2\pi}{L}M}{} \frac{e^{-iK \cdot R_I}\widehat{\psi}_K^* e^{i K' \cdot R_I }\widehat{\psi}_{K'} }{ |K-K'|^2} = \frac{8 {\pi}}{|\Gamma|^{1/2}} \sumlim{J=1}{N_\mathrm{at}} Z_J \sumlim{|K'|>\frac{2\pi}{L}M}{} \sumlim{|K|\leq \frac{2\pi}{L}M}{} \frac{e^{-iK \cdot R_I} \widehat{\psi}_K^* e^{i K' \cdot R_I } e^{-i K' \cdot R_J} }{|K'|^4 |K-K'|^2} \\
+ \sumlim{|K'|>\frac{2\pi}{L}M}{} \sumlim{|K|\leq \frac{2\pi}{L}M}{} \frac{e^{-iK \cdot R_I} \widehat{\psi}_K^* e^{i K' \cdot R_I }\eta_{K'} }{|K-K'|^2},
\end{multline*}
where in an abuse of notation we have included the $o\left(\frac{1}{|K'|^n} \right)$ in $\eta_{K'}$. 

The second double sum can be rewritten
\[
\sumlim{|K'|>\frac{2\pi}{L}M}{} \sumlim{|K|\leq \frac{2\pi}{L}M}{} \frac{e^{-iK \cdot R_I} \widehat{\psi}_K^* e^{i K' \cdot R_I }\eta_{K'} }{|K-K'|^2} = \int_\Gamma V_\mathrm{per} \Pi^\perp_M \eta \Pi_M \psi.
\]
The operator $\Pi_M V_\mathrm{per} \Pi_M^\perp \ : \ H^{-1/2-\varepsilon}_\mathrm{per}(\Gamma) \to H^{-5/2+\varepsilon}_\mathrm{per}(\Gamma)$ is continuous for all $\varepsilon >0$. 
Let $f \in H^{5/2-\varepsilon}_\mathrm{per}(\Gamma)$ and $g \in H^{1/2+\varepsilon}_\mathrm{per}(\Gamma)$, then we have
\begin{align*}
\langle f, \Pi_M V_\mathrm{per} \Pi_M^\perp g \rangle_{H^{5/2-\varepsilon},H^{-5/2+\varepsilon}} & \leq \langle \Pi_M f,  V_\mathrm{per} \Pi_M^\perp g \rangle_{H^{5/2-\varepsilon},H^{-5/2+\varepsilon}} \\
& \leq \| f \|_{H^{5/2-\varepsilon}} \| V_\mathrm{per} \Pi_M^\perp g \|_{H^{-5/2+\varepsilon}} \\
& \leq \| f \|_{H^{5/2-\varepsilon}} \| V_\mathrm{per} \|_{H^{-1/2-\varepsilon}} \| \Pi_M^\perp g \|_{H^{-2}} \\
& \leq \frac{\| V_\mathrm{per} \|_{H^{-1/2-\varepsilon}}}{M^{3/2-\varepsilon}} \| f \|_{H^{5/2-\varepsilon}}  \| g \|_{H^{-1/2-\varepsilon}}.
\end{align*}
Thus, using item 1 of Lemma~\ref{lem:basic_estimate_psi_M} applied to $\eta \in H^{7/2-\varepsilon/2}_\mathrm{per}(\Gamma)$ for all $\varepsilon>0$, we have for a constant $C_\varepsilon$ independent of $M$
\begin{align*}
\left| \int_\Gamma V_\mathrm{per} \Pi^\perp_M \eta \Pi_M \psi \right| & \leq \frac{C_\varepsilon}{M^{3/2-\varepsilon/2}} \| \Pi_M^\perp \eta \|_{H^{-1/2-\varepsilon/2}} \| \Pi_M \psi \|_{H^{5/2-\varepsilon/2}} \\
& \leq \frac{C_\varepsilon}{M^{11/2-\varepsilon}} \|\eta\|_{H^{7/2-\varepsilon/2}} \|\psi\|_{H^{5/2-\varepsilon/2}} \\
& \leq \frac{C_\varepsilon}{M^{11/2-\varepsilon}}.
\end{align*}

It remains to show that for all $1 \leq I,J \leq N_\mathrm{at}$, there exists a constant $C_\varepsilon$ such that for all $M>0$, we have
\begin{equation}
\label{eq:division_of_double_sum}
\left| \sumlim{|K'|>\frac{2\pi}{L} M}{} \sumlim{|K| \leq \frac{2 \pi}{L}M}{} \frac{\widehat{\psi}_K^* e^{-iK \cdot R_I} e^{iK' \cdot (R_I-R_J)}}{|K'|^4 |K-K'|^2}  = \sumlim{|K| \leq \frac{2 \pi}{L}M}{} \widehat{\psi}_K^* e^{-iK \cdot R_I} \sumlim{|K'|>\frac{2\pi}{L} M}{} \frac{e^{i K'\cdot (R_I-R_J)}}{|K'|^6} \right| \leq  \frac{C_\varepsilon}{M^{4-\varepsilon}} ,
\end{equation}
to complete the proof. 

We have
\begin{align}
\left| \sumlim{|K'|>\frac{2\pi}{L}M}{} \sumlim{|K|\leq \frac{2\pi}{L}M}{} \right.  & \left. \frac{\widehat{\psi}_K e^{i K' \cdot(R_I-R_J) } e^{-iK \cdot R_I}}{|K'|^4 |K-K'|^2} - \sumlim{|K|\leq \frac{2\pi}{L}M}{} \widehat{\psi}_K e^{-iK \cdot R_I} \sumlim{|K'| > \frac{2\pi}{L}M}{} \frac{e^{i K' \cdot(R_I-R_J) } }{|K'|^6 } \right| \nonumber \\
& \leq \sumlim{|K'|>\frac{2\pi}{L}M}{} \sumlim{|K|\leq \frac{2\pi}{L}M}{} \frac{|\widehat{\psi}_K|}{|K'|^6 |K'-K|^2} \left|2 |K'||K| \cos(K,K') - |K|^2 \right| \nonumber\\
& \leq 3 \sumlim{|K'|>\frac{2\pi}{L}M}{} \sumlim{|K|\leq \frac{2\pi}{L}M}{} \frac{|\widehat{\psi}_K||K| }{|K'|^5 |K'-K|^2} \nonumber \\
& \leq 3 \left(  \sumlim{|K|\leq \frac{2\pi}{L}M}{} |K|^{5-\varepsilon} |\widehat{\psi}_K|^2  \sumlim{|K'|>\frac{2\pi}{L}M}{} \frac{1}{|K'|^5} \right)^{1/2} \left(  \sumlim{|K|\leq \frac{2\pi}{L}M}{} \sumlim{|K'|>\frac{2\pi}{L}M}{} \frac{1}{|K|^{3-\varepsilon} |K-K'|^4 |K'|^5}  \right)^{1/2} \nonumber\\
& \leq \frac{C_\varepsilon}{M}  \left(  \sumlim{|K|\leq \frac{2\pi}{L}M}{} \sumlim{|K'|>\frac{2\pi}{L}M}{} \frac{1}{|K|^{3-\varepsilon} |K-K'|^4 |K'|^5}  \right)^{1/2}. \label{eq:last_estimate}
\end{align}
For clarity, we set $L=2\pi$. We split the double sum into 3 parts:
\begin{multline*}
\sumlim{|K'|> M}{} \sumlim{|K|\leq  M}{} = \sumlim{|K'|> M  }{} \sumlim{|K| \leq  M-M^{\alpha_0}}{} 
+ \sumlim{|K'|> M+M^{\alpha_0} }{} \sumlim{ M-M^{\alpha_0}  \leq |K| \leq  M}{} + \sumlim{ M < |K'| \leq M+M^{\alpha_0}}{} \sumlim{M - M^{\alpha_0} < |K|\leq M }{},
\end{multline*}
where $0 < \alpha_0 < 1$ will be determined later. 

For the first and the second double sum, using that $|K-K'|\geq M^{\alpha_0}$, we have
\begin{align}
 \sumlim{|K'|> M  }{} \sumlim{|K| \leq M-M^{\alpha_0}}{} \frac{1}{|K|^{3-\varepsilon} |K-K'|^4 |K'|^5} & \leq \frac{C_\varepsilon}{M^{2+4\alpha_0-\varepsilon}} \label{eq:division1},
\end{align}
and 
\begin{align}
 \sumlim{|K'| > M+M^{\alpha_0}}{} \sumlim{M-M^{\alpha_0}  \leq |K| \leq  M}{}  \frac{1}{|K|^{3-\varepsilon} |K-K'|^4 |K'|^5} & \leq \frac{C_\varepsilon}{M^{3+3\alpha_0-\varepsilon}} \label{eq:division2},
\end{align}
where $C_\varepsilon$ is a constant independent of $M$.

The last sum is split into several double sums
\begin{equation}
\label{eq:division_master}
\begin{split}
\sumlim{ M < |K'| \leq M+M^{\alpha_0}}{} \sumlim{M - M^{\alpha_0} < |K|\leq M }{} = \sumlim{M - M^{\alpha_0} < |K|\leq M}{} \sumlim{ \substack{M < |K'| \leq M+M^{\alpha_0} \\ |K-K'| \geq M^{\alpha_0}} }{} + \sumlim{M - M^{\alpha_0} < |K|\leq M}{} \sumlim{ \substack{M < |K'| \leq M+M^{\alpha_0} \\ M^{\alpha_1} \leq |K-K'| < M^{\alpha_0}} }{} \\
+ \sumlim{M - M^{\alpha_1} < |K|\leq M}{} \sumlim{ \substack{M < |K'| \leq M+M^{\alpha_1} \\ M^{\alpha_2} \leq |K-K'| < M^{\alpha_1}} }{} + \qquad \cdots \qquad +  \sumlim{M - M^{\alpha_{n-1}} < |K|\leq M}{} \sumlim{ \substack{M < |K'| \leq M+M^{\alpha_{n-1}} \\ M^{\alpha_n} \leq |K-K'| < M^{\alpha_{n-1}}} }{} \\
 + \sumlim{M - M^{\alpha_{n}} < |K|\leq M}{} \sumlim{ \substack{M < |K'| \leq M+M^{\alpha_{n}} \\  |K-K'| \leq M^{\alpha_{n}}} }{},
\end{split}
\end{equation}
where $0 < \alpha_n < \dots < \alpha_0 < 1$ will be determined later. 
The first term in \eqref{eq:division_master} can be estimated by
\begin{equation}
    \label{eq:1st_condition}
    \sumlim{M - M^{\alpha_0} < |K|\leq M}{} \sumlim{ \substack{M < |K'| \leq M+M^{\alpha_0} \\ |K-K'| \geq M^{\alpha_0}} }{} \frac{1}{|K|^{3-\varepsilon} |K-K'|^4 |K'|^5} \leq \frac{C_{\alpha_0,M_0}}{M^{4+2\alpha_0-\varepsilon}},
\end{equation}
where $C_{\alpha_0,M_0}$ is a constant independent of $M$.
The last term in \eqref{eq:division_master} can be bounded by
\begin{equation}
\label{eq:last_condition}
    \sumlim{M - M^{\alpha_{n}} < |K|\leq M}{} \sumlim{ \substack{M < |K'| \leq M+M^{\alpha_{n}} \\  |K-K'| \leq M^{\alpha_{n}}} }{}\frac{1}{|K|^{3-\varepsilon} |K-K'|^4 |K'|^5} \leq \frac{C_{\alpha_n,M_0}}{M^{6-4\alpha_n-\varepsilon}}.
\end{equation}
For $k=1,\dots,n$, we have
\begin{equation}
\label{eq:remaining_condition}
    \sumlim{M - M^{\alpha_{k-1}} < |K|\leq M}{} \sumlim{ \substack{M < |K'| \leq M+M^{\alpha_{k-1}} \\ M^{\alpha_k} \leq |K-K'| < M^{\alpha_{k-1}}} }{}\frac{1}{|K|^{3-\varepsilon} |K-K'|^4 |K'|^5} \leq \frac{C_{\alpha_k,M_0}}{M^{6+4\alpha_k-4\alpha_{k-1}-\varepsilon}}.
\end{equation}
We will show that the sequence $(\alpha_k)_{0 \leq k \leq n}$ obtained by setting the exponents in \eqref{eq:1st_condition} to \eqref{eq:remaining_condition} to be equal, \emph{i.e.}
\[
4+2\alpha_0-\varepsilon = 6+4\alpha_1-4\alpha_{0}-\varepsilon = \dots = 6+4\alpha_n-4\alpha_{n-1}-\varepsilon  = 6-4\alpha_n-\varepsilon,
\]
satisfies 
\begin{itemize}
    \item for all $n$, $(\alpha_k)_{0 \leq k \leq n}$ is a decreasing sequence with $\alpha_0 <1$ and $\alpha_n>0$;
    \item $\lim\limits_{n \to \infty} \alpha_0 = 1$, hence the exponent $\lim\limits_{n \to ]infty} 4+2\alpha_0-\varepsilon = 6-\varepsilon$.
\end{itemize}
The condition on the exponents can be equivalently written
\[
\left\lbrace \begin{aligned}
3 \alpha_0 - 2 \alpha_1 & = 1 && \\
\alpha_k -2\alpha_{k+1} + \alpha_{k+2} & = 0 && \forall \, 0 \leq k \leq n-2 \\
\alpha_{n-1} - 2 \alpha_n & = 0 &&
\end{aligned} \right.
\]
The sequence $(\alpha_k)_{0 \leq k \leq n}$ is a solution to a homogeneous linear difference equation, hence there exist $a,b \in \R$ such that $\alpha_k = a+bk$. Using the other conditions, we have $a=\frac{n+1}{n+3}$ and $b=-\frac{1}{n+3}$, thus $\alpha_k = \frac{n+1-k}{n+3}$. So the sequence $(\alpha_k)$ is decreasing with $\alpha_0 <1$ and $\alpha_n>0$ and $\lim\limits_{n \to \infty} \alpha_0 = 1$.

Hence, for all $\varepsilon>0$ and $M_0>0$, by choosing $n$ sufficiently large such that $2+4\alpha_0 \geq 6 - \varepsilon$, for all $M\geq M_0$, we have
\[
\sumlim{|K|\leq \frac{2\pi}{L}M}{} \sumlim{|K'|>\frac{2\pi}{L}M}{} \frac{1}{|K|^{3-\varepsilon} |K-K'|^4 |K'|^5} \leq \frac{C_{\varepsilon,M_0}}{M^{6-2\varepsilon}}.
\]
Inserting this estimate in \eqref{eq:last_estimate} finishes the proof of the lemma.
\end{proof}

\begin{lem}
\label{lem:abel}
Let $1\leq I \not=J \leq N_\mathrm{at}$. Then, there exists a positive constant $C$ independent of $M$ such that
\begin{equation}
\left| \sumlim{|K'|>M}{} \frac{e^{i K' \cdot(R_I-R_J) } }{|K'|^6 } \right|  \leq \frac{C}{M^4} .
\end{equation}
\end{lem}

\begin{proof}
For simplicity, we restrict ourself to the case $L=2\pi$. In that case, $\mathcal{R}^* = \Z^3$ and for $K \in \Z^3$, we denote $K=(k_1,k_2,k_3)$ its components.
The proof of this lemma relies on an Abel transformation and noticing that if $\theta \not= 0 \ [2\pi]$, then $\sumlim{k=n}{N} e^{ik\theta} = \frac{e^{in \theta}-e^{i(N+1)\theta} }{1-e^{i\theta}}$ is bounded independently of $n$ and $N$. 

Without loss of generality, we can assume that $(R_I-R_J) \cdot (1,0,0)^T \not= 0$. In the following, we denote by $\theta = R_I-R_J$. We have
\begin{align*}
\sumlim{|K|>M}{} \frac{e^{iK \cdot \theta}}{|K|^6} & = \sumlim{|K|>M^2}{}\frac{e^{iK \cdot \theta}}{|K|^6} + \sumlim{M <|K| \leq M^2}{} \frac{e^{iK \cdot \theta}}{|K|^6} \\
& = \sumlim{|K|>M^2}{}\frac{e^{iK \cdot \theta}}{|K|^6} + \sumlim{(k_2,k_3) \in B_M}{} \sumlim{k_1=\left\lfloor\sqrt{M^2-k_2^2-k_3^2} \right\rfloor}{\left\lfloor\sqrt{M^4-k_2^2-k_3^2} \right\rfloor} \frac{e^{iK \cdot \theta}}{|K|^6} + \sumlim{(k_2,k_3) \in B_{M^2} \setminus B_M}{} \sumlim{k_1=0}{\left\lfloor\sqrt{M^4-k_2^2-k_3^2} \right\rfloor} \frac{e^{iK \cdot \theta}}{|K|^6},
\end{align*}
where $B_R$ is the two-dimensional ball of radius $R$ and origin $0$. The sums are estimated separately.

First, using a sum-integral comparison, we have
\[
\left| \sumlim{|K|>M^2}{}\frac{e^{iK \cdot \theta}}{|K|^6} \right| \leq \sumlim{|K|>M^2}{}\frac{1}{|K|^6} \leq \int_{|x|>M^2-1} \frac{\mathrm{d}x}{|x|^6} \leq \frac{C}{M^6},
\]
where $C$ is a constant independent of $M$.

For the second sum,  using that $e^{iK \cdot \theta} = \sumlim{\ell=k_1}{M^4} e^{iK_\ell \cdot \theta} -\sumlim{\ell=k_1+1}{M^4} e^{iK_\ell \cdot \theta}$, where $K_\ell = (\ell,k_2,k_3)$, we have
\begin{align}
\sumlim{(k_2,k_3) \in B_M}{} \sumlim{k_1=\left\lfloor\sqrt{M^2-k_2^2-k_3^2} \right\rfloor}{\left\lfloor\sqrt{M^4-k_2^2-k_3^2} \right\rfloor} \frac{e^{iK \cdot \theta}}{|K|^6} & = \sumlim{(k_2,k_3) \in B_M}{} \sumlim{k_1=\left\lfloor\sqrt{M^2-k_2^2-k_3^2} \right\rfloor}{\left\lfloor\sqrt{M^4-k_2^2-k_3^2} \right\rfloor} \frac{1}{|K|^6}  \sumlim{\ell=k_1}{M^4} e^{iK_\ell \cdot \theta} \nonumber \\
& \qquad \qquad -\sumlim{(k_2,k_3) \in B_M}{} \sumlim{k_1=\left\lfloor\sqrt{M^2-k_2^2-k_3^2} \right\rfloor}{\left\lfloor\sqrt{M^4-k_2^2-k_3^2} \right\rfloor} \frac{1}{|K|^6}  \sumlim{\ell=k_1+1}{M^4} e^{iK_\ell \cdot \theta} \nonumber\\
& = \sumlim{(k_2,k_3) \in B_M}{} \frac{1}{\left|\left(\left\lfloor \sqrt{M^2-k_2^2-k_3^2} \right\rfloor,k_2,k_3\right)\right|^6} \sumlim{\ell=\left\lfloor\sqrt{M^2-k_2^2-k_3^2} \right\rfloor}{M^4} e^{iK_\ell \cdot \theta} \nonumber\\
& \qquad -  \sumlim{(k_2,k_3) \in B_M}{} \frac{1}{\left|\left(\left\lfloor \sqrt{M^4-k_2^2-k_3^2} \right\rfloor+1,k_2,k_3\right)\right|^6} \sumlim{\ell=\left\lfloor\sqrt{M^4-k_2^2-k_3^2} \right\rfloor+1}{M^4} e^{iK_\ell \cdot \theta} \label{eq:ugly_sum_after_abel} \\
& \qquad \qquad + \sumlim{(k_2,k_3) \in B_M}{} \sumlim{k_1=\left\lfloor\sqrt{M^2-k_2^2-k_3^2} \right\rfloor+1}{\left\lfloor\sqrt{M^4-k_2^2-k_3^2} \right\rfloor} \left( \frac{1}{|K|^6} - \frac{1}{|K_{k_1-1}|^6} \right) \sumlim{\ell=k_1}{M^4}  e^{i K_\ell \cdot \theta}.\nonumber
\end{align}
The first term in \eqref{eq:ugly_sum_after_abel} can be estimated as follows. $\frac{1}{\left|\left(\left\lfloor \sqrt{M^2-k_2^2-k_3^2} \right\rfloor,k_2,k_3\right)\right|^6}$ is equivalent to $\frac{1}{M^6}$ and $\sumlim{\ell=\left\lfloor\sqrt{M^2-k_2^2-k_3^2} \right\rfloor}{M^4} e^{iK_\ell \cdot \theta}$ can be bounded independently of $M$. Hence, there is a constant $C$ independent of $M$ such that
\[
\left| \sumlim{(k_2,k_3) \in B_M}{} \frac{1}{\left|\left(\left\lfloor \sqrt{M^2-k_2^2-k_3^2} \right\rfloor,k_2,k_3\right)\right|^6} \sumlim{\ell=\left\lfloor\sqrt{M^2-k_2^2-k_3^2} \right\rfloor}{M^4} e^{iK_\ell \cdot \theta} \right| \leq \frac{C}{M^4}.
\]
The second term in \eqref{eq:ugly_sum_after_abel} can be treated the same way. Finally, noticing that $\frac{1}{|K|^6} - \frac{1}{|K_{k_1-1}|^6}=\cO \left(\frac{1}{|K|^7} \right)$, we deduce that for a constant $C$ independent of $M$, we have
\[
\left| \sumlim{(k_2,k_3) \in B_M}{} \sumlim{k_1=\left\lfloor\sqrt{M^2-k_2^2-k_3^2} \right\rfloor+1}{\left\lfloor\sqrt{M^4-k_2^2-k_3^2} \right\rfloor+1} \left( \frac{1}{|K|^6} - \frac{1}{|K_{k_1-1}|^6} \right) \sumlim{\ell=k_1}{M^4-1}  e^{i K_\ell \cdot \theta} \right| \leq C \sumlim{|K|\geq M}{} \frac{1}{|K|^7} \leq \frac{C}{M^4}.
\]
This finishes the proof of the lemma.
\end{proof}

We have now all the necessary tools to prove Theorem \ref{thm:error_cancellation}.

\begin{proof}[Proof of Theorem~\ref{thm:error_cancellation}]
Let $E_M$ be an eigenvalue of the variational approximation \eqref{eq:discrete_eigenvalue_problem} and $E$ the corresponding exact eigenvalue. 
Let $\psi$ be an eigenfunction associated to $E$. 
By Lemma \ref{lem:E_M-E} and Lemma \ref{lem:remainder_explicited}, we know that 
\[
E_M - E =  -\frac{4 \pi}{|\Gamma|}\sumlim{I=1}{N_\mathrm{at}} Z_I \psi(R_I) \sumlim{|K'|> \frac{2 \pi}{L} M}{} \sumlim{|K| \leq \frac{2 \pi}{L} M}{} Z_I \frac{e^{-i K \cdot R_I} e^{iK' \cdot R_I}}{|K'-K|^2} \widehat{\psi}_K^*\widehat{\psi}_K' \left(1+ \cO \left( \frac{1}{M^{5/2-\varepsilon}} \right) \right)+ \cO \left( \frac{1}{M^{9/2-\varepsilon}} \right).
\]
By Lemma \ref{lem:double_sum_split}, for all $\varepsilon>0$ and $M_0>0$, there exists a constant $C_{\varepsilon,M_0}$ such that for all $M \geq M_0$ we have
\begin{align*}
\Bigg| E_M -E  - \frac{32 \pi^{2} }{|\Gamma|^{3/2}} &  \sumlim{I=1}{N_\mathrm{at}} Z_I \psi(R_I)\sumlim{J=1}{N_\mathrm{at}} Z_J \sumlim{|K|\leq \frac{2\pi}{L}M}{} \widehat{\psi}_K^* e^{-iK \cdot R_I} \sumlim{|K'| > \frac{2\pi}{L}M}{} \frac{e^{i K' \cdot(R_I-R_J) } }{|K'|^6 }  \left(1+ \cO \left( \frac{1}{M^{5/2-\varepsilon}} \right) \right) \Bigg| \\
& \leq \frac{C_{\varepsilon, M_0}}{M^{4-\varepsilon}}  
\end{align*}
Using Lemma \ref{lem:abel}, we get for all $M \geq M_0$
\begin{align*}
\Bigg| E_M -E - \frac{32 \pi^{2}}{|\Gamma|^{3/2}} & \sumlim{I=1}{N_\mathrm{at}} Z_I^2 \psi(R_I)\sumlim{|K|\leq \frac{2\pi}{L}M}{} \widehat{\psi}_K^* e^{-iK \cdot R_I} \sumlim{|K'| > \frac{2\pi}{L}M}{} \frac{1}{|K'|^6 }  \left(1+ \cO \left( \frac{1}{M^{5/2-\varepsilon}} \right) \right) \Bigg|  \leq \frac{C_{\varepsilon, M_0}}{M^{4-\varepsilon}}.
\end{align*}
We have
\[
\frac{1}{|\Gamma|^{1/2}} \sumlim{I=1}{N_\mathrm{at}} Z_I^2 \sumlim{|K|\leq \frac{\pi}{L}M}{} \widehat{\psi}_K^* e^{-iK \cdot R_I} = \sumlim{I=1}{N_\mathrm{at}} Z_I^2 (\Pi_M \psi)(R_I)^* = \sumlim{I=1}{N_\mathrm{at}} Z_I^2  \psi(R_I) + \cO \left( \frac{1}{M^{1-\varepsilon}} \right),
\]
since $H^{3/2+\varepsilon}_\mathrm{per}(\Gamma) \hookrightarrow C^{0}_\mathrm{per}(\Gamma) $ for all $\varepsilon>0$.
Moreover, by a sum-integral comparison, we deduce that
\[
\sumlim{|K'| > \frac{2\pi}{L}M}{} \frac{1}{|K'|^6 } = \frac{L^6}{(2\pi)^6} \sumlim{|(k_1,k_2,k_3)| > M}{} \frac{1}{|(k_1,k_2,k_3)|^6} = \frac{L^6}{16\pi^5} \frac{1}{3M^3} + \cO \left( \frac{1}{M^4} \right).
\]
Hence, for all $M \geq M_0$
\begin{align*}
\left| E_M - E - \frac{2L^{3}}{3 \pi^{3} M^3} \sumlim{I=1}{N_\mathrm{at}} Z_I^2 \psi(R_I)^2 \right| \leq \frac{C_{\varepsilon,M_0}}{M^{4-\varepsilon}} .
\end{align*}
\end{proof}

\section{Numerical tests}
\label{sec:numeric}

In this section, we present some numerical results for the Hamiltonian $H$ 
\begin{equation}
\label{eq:numerical_hamiltonian}
H = -\frac{1}{2} \Delta - \frac{Z}{\left|x-\frac{{R}}{2}\right|} - \frac{Z}{\left|x+\frac{{R}}{2}\right|} ,
\end{equation}
with periodic boundary conditions on the unit cell $[-\tfrac{L}{2},\tfrac{L}{2}]^3$ for $L=2$.

The eigenvalue problem $H \psi = E\psi$ for $H$ defined above is solved using a slightly different plane-wave discretization space 
\begin{equation*}
V_M = \left\lbrace \sumlim{|K|_\infty \leq \frac{2\pi}{L}M }{} c_K e_K \ \Big| \ c_{-K}=c_K^* \right\rbrace ,
\end{equation*}
where for $K=(k_1,k_2,k_3)$, $|K|_\infty = \max(|k_1|,|k_2|,|k_3|)$. Theorem~\ref{thm:error_cancellation} still holds, but with another error cancellation prefactor since
\[
\int_{\R^3 \setminus  \lbrace |x|_\infty \leq 1 \rbrace} \frac{1}{(x_1^2+x_2^2+x_3^2)^3} \, \mathrm{d}x_1 \, \mathrm{d}x_2 \, \mathrm{d}x_3 = \frac{1}{3} + \frac{5 \sqrt{2}}{2} \arctan(\tfrac{1}{\sqrt{2}})  =:A.
\]
Hence, the error estimate \eqref{eq:final_error_estimate} becomes
\begin{equation}
\label{eq:error_correction_cubic}
E_M -E = \frac{A L^3 Z^2}{2 \pi^4 } \frac{\psi(-\tfrac{R}{2})^2 +\psi(\tfrac{R}{2})^2}{M^3} + \cO \left( \frac{1}{M^{4-\varepsilon}} \right).
\end{equation}
In the following, the prefactor is denoted by $C(R,Z,L) = \frac{A L^3 Z^2}{2 \pi^4} \left( \psi(R/2)^2 + \psi(-R/2)^2 \right)$.
\newline

\begin{figure}[h!]
	\centering
    \begin{subfigure}[b]{0.4\textwidth}
        \includegraphics[width=\textwidth]{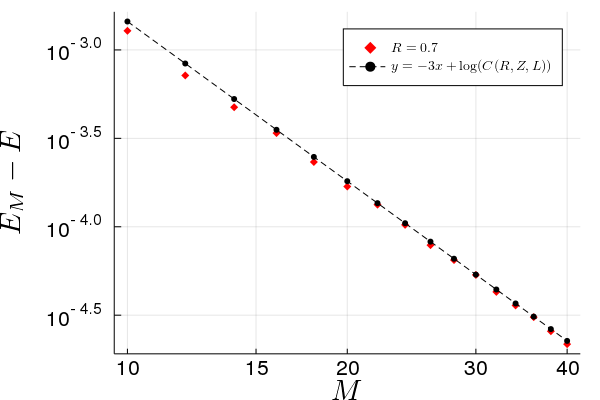}
    	\caption{$Z=2$}
    \end{subfigure}
	\quad
    \begin{subfigure}[b]{0.4\textwidth}
        \includegraphics[width=\textwidth]{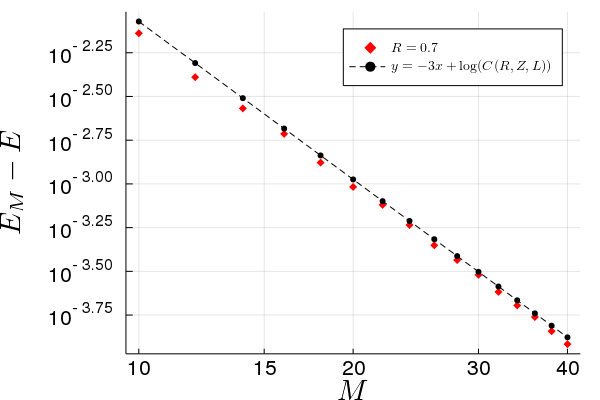}
    \caption{$Z=3$}
    \end{subfigure}
\caption{Error on the lowest eigenvalue of the plane-wave discretization}
\label{fig:error_compared_first_order}
\end{figure}

In Figure~\ref{fig:error_compared_first_order}, the discretization error $E_M-E$ is compared to the first order correction \eqref{eq:error_correction_cubic}. To compute the energy reference and the value $\psi(\tfrac{R}{2})$, the eigenvalue problem \eqref{eq:numerical_hamiltonian} is solved using a sufficiently large plane-wave cutoff ($M=100$). 
We notice a very good agreement between the error on the eigenvalue and the first order correction.

In Figure~\ref{fig:error_cancellation}, the discretization error cancellation for our simple model \eqref{eq:numerical_hamiltonian} is highlighted.  
In this case, the discretization error on the energy difference is simply
\[
D_M = \left| (E_{R_1,M}-E_{R_2,M}) -(E_{R_1} -E_{R_2}) \right|
\]
and the sum of the discretization errors on the lowest eigenvalue for both configurations is given by
\[
S_M = \left| E_{R_1,M} -E_{R_1} \right| + \left| E_{R_2,M} -E_{R_2} \right|.
\]
We can see that the discretization error $D_M$ converges as the same rate as the sum of the discretization errors $S_M$, however a prefactor of order $|R_2-R_1|$ is gained. 
\newline

\begin{figure}[h!]
	\centering
    \begin{subfigure}[b]{0.4\textwidth}
    	\includegraphics[width=\textwidth]{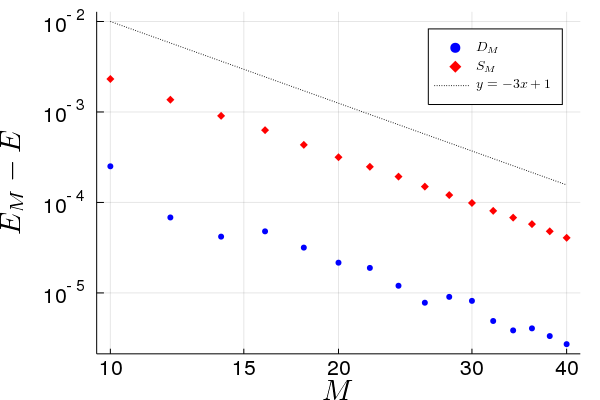}
        \caption{$Z=2$}
    \end{subfigure}
	\quad
    \begin{subfigure}[b]{0.4\textwidth}
    	\includegraphics[width=\textwidth]{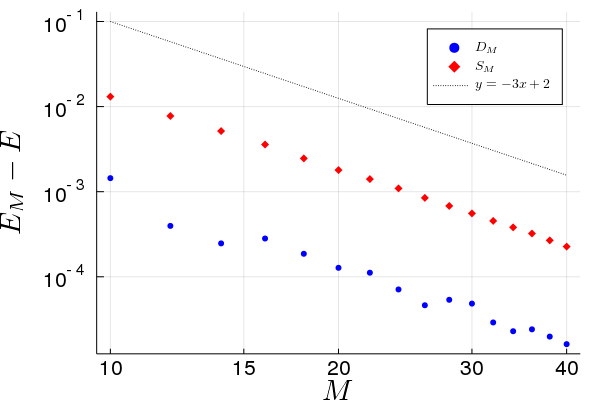}
        \caption{$Z=3$}
    \end{subfigure}
    \caption{Error on the computed energy difference $D_M$ compared to the total discretization error $S_M$.}
\label{fig:error_cancellation}
\end{figure}

In Figure~\ref{fig:error_correction_efficiency}, the error on the lowest eigenvalue with the first order correction~\eqref{eq:first_order_correction} is given. In our example, the convergence rate of the error is in accordance with the convergence rate given by Theorem~\ref{thm:error_cancellation}.

\begin{figure}[H]
	\centering
    \begin{subfigure}[b]{0.4\textwidth}
        \includegraphics[width=\textwidth]{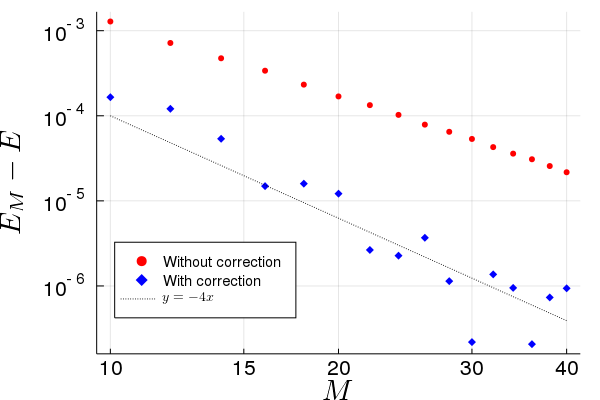}
    	\caption{$Z=2,\ R=0.7e_1$}
    \end{subfigure}
	\quad
    \begin{subfigure}[b]{0.4\textwidth}
        \includegraphics[width=\textwidth]{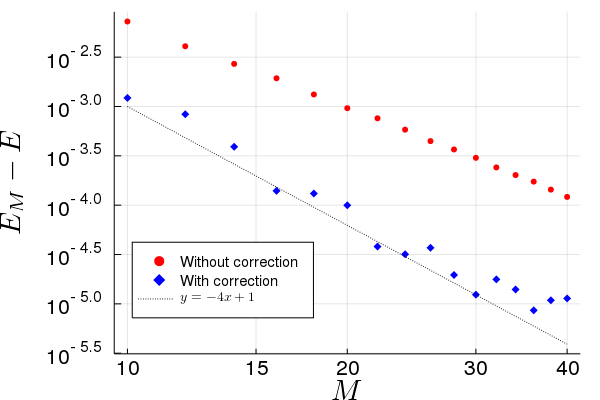}
    \caption{$Z=3,\ R=0.7e_1$}
    \end{subfigure}
    \caption{Error on the lowest eigenvalue of the plane-wave discretization with and without the first order correction.}
\label{fig:error_correction_efficiency}
\end{figure}

\section{Conclusion}

In the present work, discretization error cancellation has been analyzed for plane-wave discretization of a periodic Hamiltonian with Coulomb potentials. 
Using weighted Sobolev spaces and singular expansions of the eigenfunctions, an explicit formula of the leading term of the plane-wave convergence of the eigenvalue is proved. 
This yields a precise characterization of the error cancellation in agreement with our numerical tests. 

In \cite{cances2017discretization}, numerical tests of plane-wave methods on periodic Hamiltonian with pseudopotentials have suggested that error cancellation also occurs. 
The same analysis may be reproduced in that case, if the singularities on the eigenfunctions
due to the pseudopotential can be characterized.
It would be interesting to investigate the discretization error cancellation in other settings, for different models or different families of basis sets (finite elements, \dots). 

\newcommand{\etalchar}[1]{$^{#1}$}


\begin{thebibliography}{CDM{\etalchar{+}}16}

\bibitem[BO89]{babuska1989finite}
Ivo Babu{\v s}ka and John~E. Osborn.
\newblock Finite element-{G}alerkin approximation of the eigenvalues and
  eigenvectors of selfadjoint problems.
\newblock {\em Math. Comp.}, 52(186):275--297, 1989.

\bibitem[BR73]{babuska1972finite}
Ivo Babu{\v s}ka and Michael~B. Rosenzweig.
\newblock A finite element scheme for domains with corners.
\newblock {\em Numer. Math.}, 20:1--21, 1972/73.

\bibitem[CCM12]{cances2012numericalPWdiscretization}
Eric Canc\`es, Rachida Chakir, and Yvon Maday.
\newblock Numerical analysis of the planewave discretization of some
  orbital-free and {K}ohn-{S}ham models.
\newblock {\em ESAIM Math. Model. Numer. Anal.}, 46(2):341--388, 2012.

\bibitem[CD17]{cances2017discretization}
Eric Canc\`es and Genevi\`eve Dusson.
\newblock Discretization error cancellation in electronic structure
  calculation: toward a quantitative study.
\newblock {\em ESAIM Math. Model. Numer. Anal.}, 51(5):1617--1636, 2017.

\bibitem[CDM{\etalchar{+}}16]{cances2016perturbation}
Eric Canc\`es, Genevi\`eve Dusson, Yvon Maday, Benjamin Stamm, and Martin
  Vohral\'ik.
\newblock A perturbation-method-based post-processing for the planewave
  discretization of {K}ohn-{S}ham models.
\newblock {\em J. Comput. Phys.}, 307:446--459, 2016.

\bibitem[CGH{\etalchar{+}}13]{chen2013numerical}
Huajie Chen, Xingao Gong, Lianhua He, Zhang Yang, and Aihui Zhou.
\newblock Numerical analysis of finite dimensional approximations of
  {K}ohn-{S}ham models.
\newblock {\em Adv. Comput. Math.}, 38(2):225--256, 2013.

\bibitem[CS15]{chen2015numerical}
Huajie Chen and Reinhold Schneider.
\newblock Numerical analysis of augmented plane wave methods for full-potential
  electronic structure calculations.
\newblock {\em ESAIM Math. Model. Numer. Anal.}, 49(3):755--785, 2015.

\bibitem[Dup18]{dupuy2018vpaw3d}
Mi-Song Dupuy.
\newblock The variational projector augmented-wave method for the plane-wave
  discretization of linear {S}chr\"odinger operators.
\newblock {\em in preparation}, 2018.

\bibitem[ES97]{egorov2012pseudo}
Yuri~V. Egorov and Bert-Wolfgang Schulze.
\newblock {\em Pseudo-differential operators, singularities, applications},
  volume~93 of {\em Operator Theory: Advances and Applications}.
\newblock Birkh\"auser Verlag, Basel, 1997.

\bibitem[FSS08]{flad2008asymptotic}
Heinz-J\"urgen Flad, Reinhold Schneider, and Bert-Wolfgang Schulze.
\newblock Asymptotic regularity of solutions to {H}artree-{F}ock equations with
  {C}oulomb potential.
\newblock {\em Math. Methods Appl. Sci.}, 31(18):2172--2201, 2008.

\bibitem[Gri92]{grisvard1992singularities}
Pierre Grisvard.
\newblock {\em Singularities in boundary value problems}, volume~22 of {\em
  Recherches en Math\'ematiques Appliqu\'ees [Research in Applied
  Mathematics]}.
\newblock Masson, Paris; Springer-Verlag, Berlin, 1992.

\bibitem[HNS08]{hunsicker2008analysis}
Eugenie Hunsicker, Victor Nistor, and Jorge~O. Sofo.
\newblock Analysis of periodic {S}chr\"odinger operators: regularity and
  approximation of eigenfunctions.
\newblock {\em J. Math. Phys.}, 49(8):083501, 21, 2008.

\bibitem[Kat57]{kato1957eigenfunctions}
Tosio Kato.
\newblock On the eigenfunctions of many-particle systems in quantum mechanics.
\newblock {\em Comm. Pure Appl. Math.}, 10:151--177, 1957.

\bibitem[KMR97]{kozlov1997elliptic}
Vladimir~A. Kozlov, Vladimir~G. Maz'ya, and J\"urgen Rossmann.
\newblock {\em Elliptic boundary value problems in domains with point
  singularities}, volume~52 of {\em Mathematical Surveys and Monographs}.
\newblock American Mathematical Society, Providence, RI, 1997.

\bibitem[Mel93]{melrose93}
Richard~B. Melrose.
\newblock {\em The {A}tiyah-{P}atodi-{S}inger index theorem}, volume~4 of {\em
  Research Notes in Mathematics}.
\newblock A K Peters, Ltd., Wellesley, MA, 1993.

\bibitem[PCH08]{pieniazek2008sources}
Susan~N. Pieniazek, Fernando~R. Clemente, and Kendall~N. Houk.
\newblock Sources of {E}rror in {DFT} {C}omputations of {C-C B}ond {F}ormation
  {T}hermochemistries: $\pi \rightarrow \sigma$ {T}ransformations and {E}rror
  {C}ancellation by {DFT M}ethods.
\newblock {\em Angewandte Chemie International Edition}, 47(40):7746--7749,
  2008.

\bibitem[RS78]{reed1978iv}
Michael Reed and Barry Simon.
\newblock {\em Methods of modern mathematical physics. {IV}. {A}nalysis of
  operators}.
\newblock Academic Press [Harcourt Brace Jovanovich, Publishers], New
  York-London, 1978.
\end{thebibliography}
\end{document}